\documentclass[12pt,reqno]{article}
\usepackage[margin=18mm]{geometry}
\usepackage{amssymb,amsmath,amsthm,latexsym, bbm,thmtools, dsfont}
\usepackage{xcolor}
\usepackage{enumerate}
\usepackage{graphicx}
\usepackage{mathrsfs}
\usepackage[mathscr]{euscript}
\usepackage{amscd}
\usepackage[english]{babel}
\usepackage{stmaryrd}
\usepackage[titletoc]{appendix}
\usepackage[colorlinks=true, linkcolor=blue, citecolor=blue]{hyperref}
\allowdisplaybreaks

\usepackage[backgroundcolor=white, bordercolor=blue,
linecolor=blue]{todonotes}
\usepackage[normalem]{ulem}

\usepackage[notref,notcite]{showkeys}

\newtheorem{thm}{Theorem}[section]

\newtheorem{lemma}[thm]{Lemma}
\newtheorem*{theorem*}{Theorem}

\theoremstyle{definition}
\newtheorem{defn}[thm]{Definition}

\newtheorem{remark}[thm]{Remark}

\newtheorem{example}[thm]{Examples}
\addto\extrasenglish{}
\addto\extrasenglish{}
\numberwithin{equation}{section}

\usepackage{color}

\def\R{\mathbb R}

\def\E{\mathscr{E}}
\def\D{\mathscr{D}}

\def\de{\end{equation}}

\def\edar{\end{eqnarray}}

\def\l{\left}\def\r{\right}

\def\Ent{\text {\rm Ent}}

\def\var{\text {\rm Var}}

\def\Ent{{\rm Ent}}

\def\[{\l[} \def\]{\r]}
\def\({\l(} \def\){\r)}

\title{\bf  {A note on the equivalence of super-Poincar\'e inequality}}

\author{Xin Chen
\qquad Qiuchen Yang}

\author{
 Xin Chen,
  \and
   Qiuchen Yang 
}

\date{}

\begin{document}
\allowdisplaybreaks
 \maketitle


\begin{abstract}
In this paper we will study the equivalence between super-Poincar\'e inequality and some log-Sobolev type inequalities, including weak log-Sobolev inequality
and super log-Sobolev inequality. The explicit relations between associated rate functions will also be established.
\end{abstract}





\section{Introduction}
Let $(\Omega,\mathscr{F},\mu)$ be a probability space, and suppose that $(\E,\D(\E))$ is a
symmetric Dirichlet form on $L^2(\mu):=L^2(\Omega;\mu)$ such that 
\begin{equation}\label{e1-2a}
\E(f,1)=0,\ \ \forall\ f\in \D(\E). 
\end{equation}

For every $f\in L^2(\mu)$, let $\mu(f):=\int_{\Omega}f d\mu$
be the integral of $f$ with respect to probability measure $\mu$. The following log-Sobolev
inequality (written as LSI) was introduced by Gross \cite{G},
\begin{equation*}
\Ent_{\mu}(f^2)\le C_0\E(f,f),\ f\in \D(\E),
\end{equation*}
where $\Ent_{\mu}(f^2)$ is
the entropy of $f^2$ with respect to $\mu$, i.e.,
\begin{equation*}
\Ent_\mu(f^2):=\mu\left(f^2\log f^2\right)-\mu(f^2)\log \mu(f^2).
\end{equation*}

The LSI has been widely applied to study various different properties, including
the convergence rate of the entropy for corresponding semi-group, the contractivity for corresponding semi-group,
the rigidity and geometric property of based manifold and so on. We refer the reader
to monographs \cite{GZ,Le,30b} for more details about the subjects related to LSI.

During the development of the theory about LSI, many variations of LSI have also been investigated.
Now we will introduce several different
types of inequalities for $\mu$ and $(\E,\D(\E))$, which can be viewed as extensions of LSI.
\begin{defn}
\begin{itemize}
\item [(1)] We say that \emph{super-Poincar\'e inequality} (written as SPI) holds for $\mu$ and $(\E,\D(\E))$ with
associated rate function $\beta_{{\rm SP}}(\cdot)$ if
\begin{equation}\label{e1-1}
\mu(f^2)\le s\E(f,f)+\beta_{{\rm SP}}(s)\mu(|f|)^2,\ s>0,\ f\in \D(\E),
\end{equation}
where $\beta_{{\rm SP}}:\R_+ \to \R_+$ is non-increasing.

\item [(2)] We say that \emph{super log-Sobolev inequality} (written as SLSI) holds for $\mu$ and $(\E,\D(\E))$
with associated rate function $\beta_{{\rm SL}}(\cdot)$ if
\begin{equation}\label{e1-2}
\Ent_{\mu}(f^2)\le s\E(f,f)+\beta_{{\rm SL}}(s)\mu(f^2),\ s>0,\ f\in \D(\E),
\end{equation}
where $\beta_{{\rm SL}}:\R_+ \to \R_+$ is non-increasing.
\end{itemize}

\item [(3)] We say that \emph{weak log-Sobolev inequality} (written as WLSI) holds for $\mu$ and $(\E,\D(\E))$
with associated rate function $\beta_{{\rm WL}}(\cdot)$ if
\begin{equation}\label{e1-3}
\Ent_{\mu}(f^2)\le \beta_{{\rm WL}}(s)\E(f,f)+s\|f\|_\infty^2,\ s>0,\ f\in \D(\E),
\end{equation}
where $\beta_{{\rm WL}}:\R_+ \to \R_+$ is non-increasing.

\end{defn}
These functional inequalities are closely related to properties
for Hunt processes and semigroups associated with $(\E,\D(\E))$. The super log-Sobolev inequality
was introduced by Davies and Simon \cite{DS84} to study contractivity properties
for associated semigroups. Wang \cite{W} investigated the super-Poincar\'e inequality and applied it
to characterize the emptiness of essential spectrum for associated infinitesimal
generators.  Corresponding theory for the weak Poincar\'e inequality
was developed by R\"ockenr and Wang \cite{RW} and it was used to establish
the (long time) algebraic convergence rates of associated semigroups. Cattiaux, Gentil and Guillin
\cite{WLSI} introduced the weak log-Sobolev inequality and studied the convergence
rates for entropies of associated semigroups.
See \cite{BCG08,B85,B05,BCR,BCR1,BM,Cao23,CG,entropy,GW19,GLM,MJT,W1,WW15} and
references therein for different applications of these functional inequalities.
We also refer the reader to monograph \cite{30b} for
detailed introduction on this subject.

Meanwhile, these functional inequalities can be viewed as interpolations of standard Poincar\'e inequalities and log-Sobolev inequalities. It is natural
to study the connection between them, especially the relation among rate functions. Using measure-capacity method,
\cite[Section 3.2]{WLSI} established a criterion for SPI based on WLSI, but the equivalence
between them was still unknown. The equivalence between SLSI, a kind
of Nash type inequality and Orlicz-Sobolev inequality was investigated in \cite{BM}.
\emph{In this paper, we will study direct relations between SPI and SLSI or WLSI, especially
those among rate functions}, which
are still not clear to the best of our knowledge.

Now we will give the main theorem of this paper.

\begin{thm}\label{t1-1}
The following statements hold.
\begin{itemize}
\item [(1)] Suppose that SPI \eqref{e1-1} holds with associated rate function
$\beta_{{\rm SP}}(\cdot)$ and the following weak Poincar\'e inequality holds,
\begin{equation}\label{l3-5-1a}
\mu(f^2)-\mu(f)^2\le \beta_{{\rm WP}}(s)\E(f,f)+s\|f\|_\infty^2,\ s>0,\ f\in \D(\E),
\end{equation}
where $\beta_{{\rm WP}}:\R_+ \to \R_+$ is non-increasing.
Then $\mu$ satisfies the WLSI \eqref{e1-3}
with the rate function $\beta_{{\rm WL}}:\R_+ \to \R_+$ as
\begin{equation}\label{l3-5-2a}
\beta_{{\rm WL}}(s):=
\begin{cases}
C_1\inf\left\{\sup_{n_0\le n \le k}n\xi_1(\delta^{-n+1});
\ k\ge n_0\ {\rm satisfies}\ C_2k\delta^{-k}\le s\right\},\ &s\in (0,s_0],\\
\beta_{{\rm WL}}(s_0),\ &s\in (s_0,+\infty).
\end{cases}
\end{equation}
Here $C_1$, $C_2$, $s_0$, $\delta>2$, $n_0\ge 2$ are some positive constants,
\begin{equation}\label{t1-2-1}
\xi_1(t):=\inf\left\{\frac{r}{1-t\beta_{{\rm SP}}(r)};r>0\ {\rm such\ that}\ 1-t\beta_{{\rm SP}}(r)>0\right\},\ t>0,
\end{equation}
and we use the convention that $\inf \emptyset=-\infty$.

\medskip
\item [(2)] Suppose that WLSI \eqref{e1-3} holds with associated rate function $\beta_{{\rm WL}}(\cdot)$,
and there exist $\theta>0$, $\delta\ge 2$ such that
\begin{equation}\label{l3-1-1}
\lim_{n \to \infty} \frac{\beta_{{\rm WL}}\left(\delta^{-n}n^{-\theta}\right)}{n}=0,
\end{equation}
then
the SPI \eqref{e1-1} is true with the following rate function,
\begin{equation}\label{l3-1-3}
\beta_{{\rm SP}}(s)=
\begin{cases}
C_3\inf\left\{\delta^k; k\ge n_0\ {\rm\ such\ that}\ \sup_{n\ge k}\frac{\beta_{{\rm WL}}\left(\delta^{-n}n^{-\theta}\right)}{n}\le s\right\},
\ &\ s\in (0,s_0],\\
\beta_{{\rm SP}}(s_0),\ &\ s\in (s_0,+\infty),
\end{cases}
\end{equation}
where $C_3$, $n_0$, $s_0$ are some positive constants.
\end{itemize}
\end{thm}

\begin{thm}\label{t1-2}
The following statements hold.
\begin{itemize}
\item [(1)] Suppose that SPI \eqref{e1-1} holds with associated rate function $\beta_{{\rm SP}}(\cdot)$ and
\begin{equation}\label{l3-5-2}
\lim_{n \to \infty}n\xi_1(\delta^{-n+1})=0,
\end{equation}
where $\xi_1(\cdot)$ is defined by \eqref{t1-2-1}.
Then the SLSI \eqref{e1-2} holds
with rate function $\beta_{{\rm SL}}(s)$ defined by
\begin{equation}\label{l3-5-3}
\beta_{{\rm SL}}(s):=
\begin{cases}
(\log \delta)(1+N_0(s)),\ &\ s\in (0,s_0],\\
\beta_{{\rm SL}}(s_0),\ &\ s\in (s_0,+\infty),
\end{cases}
\end{equation}
where
\begin{equation}\label{l3-5-3aa}
N_0(s):=\sup\left\{n\ge n_0; C_4n\xi_1\left(\delta^{-n+1}\right)> s\right\},
\end{equation}
$C_4$, $s_0$, $\delta>2$ are some positive constants.
\medskip
\item [(2)] Suppose that SLSI \eqref{e1-2} holds with associated rate function $\beta_{{\rm SL}}(\cdot)$, then
the SPI \eqref{e1-1} holds with the rate function $\beta_{{\rm SP}}(\cdot)$ as follows
\begin{equation}\label{l3-4-1}
\begin{split}
\beta_{{\rm SP}}(s):=&
C_5\inf\Bigg\{\delta^{k};k\ge n_0,\ {\rm such\ that}\
\xi_2(k \log \delta)\le C_6s\Bigg\},
\end{split}
\end{equation}
where $C_5$, $C_6$, $\delta>2$, $\gamma\in (0,1)$, $n_0\ge 2$ are some positive constants,
\begin{equation}\label{t1-2-1a}
\xi_2(t):=\inf\left\{\frac{r}{t-\beta_{{\rm SL}}(r)};r>0\ {\rm such\ that}\ t-\beta_{{\rm SL}}(r)>0\right\},
\end{equation}
and we use the convention that $\inf \emptyset=-\infty$.
\end{itemize}
\end{thm}

We will also give some remarks for our main theorems.
\begin{itemize}
\item [(i)] In Theorem \ref{t1-1} (1) we obtain an estimate for the rate function $\beta_{{\rm WL}}(\cdot)$ of WLSI from 
associated rate function $\beta_{{\rm SP}}(\cdot)$ of SPI, which
answers the question raised by \cite{WLSI} (see e.g. \cite[Page 576]{WLSI} for details).
We want to remark that although the validity of weak Poincar\'e inequality \eqref{l3-5-1a} is required
in Theorem \ref{t1-1} (1), the order of $\beta_{{\rm WL}}(\cdot)$ is independent of
the rate function $\beta_{{\rm WP}}(\cdot)$ for \eqref{l3-5-1a} (but the constant $C_1$
may depend on $\beta_{{\rm WP}}(\cdot)$), see Example \ref{ex1-1} below. In particular,
according to the proof we know that we only need \eqref{l3-5-1a}
to establish Poincar\'e inequality \eqref{l3-5-7} (which is actually ensured by \eqref{l3-5-1a} and \eqref{e1-1}).
We also refer the reader to \cite{GW} about other sufficient conditions for \eqref{l3-5-7} when
SPI \eqref{e1-1} holds.
\medskip

\item [(ii)] In Theorem \ref{t1-1} (2), a concrete estimate of $\beta_{{\rm SP}}(\cdot)$ in \eqref{e1-1} is given
through $\beta_{{\rm WL}}(\cdot)$ under the condition \eqref{l3-1-1}. Similar result, i.e., the implication of SPI from WLSI
for the local Dirichlet form on $\R^d$, has been established by \cite[Proposition 3.4]{WLSI}.
Compared with \cite[Proposition 3.4]{WLSI}, our result could cover the case for a general Dirichlet form, not only
a local Dirichlet form on $\R^d$. Moreover, our condition \eqref{l3-1-1} here is weaker than the corresponding one
in \cite[Proposition 3.4]{WLSI}, where the constant $\theta$ is chosen to be $1$.

\medskip

\item [(iii)] As mentioned above, the equivalence between SLSI and some Nash type inequalities has been studied in
\cite[Theorem 1.1]{BM}. But it seems difficult to obtain the relation for rate functions
$\beta_{{\rm SP}}(\cdot)$ and  $\beta_{{\rm SL}}(\cdot)$ directly from this result. We establish
an explicit expression between $\beta_{{\rm SP}}(\cdot)$ and  $\beta_{{\rm SL}}(\cdot)$ in Theorem \ref{t1-2}, and
a sufficient condition \eqref{l3-5-2} for SLSI based on SPI is also given. This condition is sharp
as illustrated by Example \ref{ex1-1} below.
 \end{itemize}

The following examples illustrate that our main theorems are sharp in some sense.
\begin{example}\label{ex1-1}
\begin{itemize}
\item [(1)] Suppose that SPI \eqref{e1-1} holds with rate function
\begin{equation}\label{ex1-1-1}
\beta_{{\rm SP}}(s)=\exp\left(C_1\left(1+\frac{1}{s^\theta}\right)\right),\ s>0,
\end{equation}
for some constants $C_1>0$, $\theta\in [1/2,+\infty)$.

When $\theta\in [1/2,1)$, the SLSI \eqref{e1-2} is true with the rate function
\begin{equation}\label{ex1-1-2}
\beta_{{\rm SL}}(s)=C_2(1+s^{-\frac{\theta}{1-\theta}}),\ s>0,
\end{equation}
for some constant $C_2>0$.

When $\theta\in [1,+\infty)$ and the weak Poincar\'e inequality \eqref{l3-5-1a} holds, then the WLSI \eqref{e1-3} is true with the rate function
\begin{equation}\label{ex1-1-3}
\beta_{{\rm WL}}(s)=C_3\left(1+\log^{\frac{\theta-1}{\theta}}(1+s^{-1})\right).
\end{equation}
In particular, if $\theta=1$, then standard log-Sobolev inequality holds.
\medskip

\item [(2)] Conversely, if the SLSI \eqref{e1-2} holds with rate function
$\beta_{{\rm SL}}(\cdot)$ given by \eqref{ex1-1-2} for some $\theta\in [1/2,1)$, then the SPI
\eqref{e1-1} is true with associated rate function \eqref{ex1-1-1}
(which may have a different constant $C_1$).

If the WLSI \eqref{e1-3} holds with rate function $\beta_{{\rm WL}}(\cdot)$ being \eqref{ex1-1-3} for
some $\theta\in [1,+\infty)$, then the SPI \eqref{e1-1} is true with associated rate function
defined by \eqref{ex1-1-1}.
\end{itemize}
\end{example}

\begin{remark}
Suppose that $\Omega=\R^d$ for some positive integer $d\ge 1$, $\mu$ is a probability measure
on $(\R^d,\mathscr{B}(\R^d))$ with the form $\mu(dx)=e^{-V(x)}dx$, $(\E,\D(\E))$ is a local Dirichlet form such that
\begin{equation*}
\E(f,f)=\int_{\R^d}|\nabla f(x)|^2\mu(dx),\ \ f\in C_c^1(\R^d),
\end{equation*}
and $\D(\E)=\overline{C_c^1(\R^d)}^{\sqrt{\E(\cdot)}+\|\cdot\|_{L^2(\mu)}}$.

According to \cite[Corollary 2.5]{W}, when $V(x)=-c_0|x|^{\kappa}$ for some $\kappa>1$, then SPI holds
with the following rate function
\begin{equation*}
\beta_{{\rm SP}}(s)=\exp\left(c_1\left(1+s^{-\frac{\kappa}{2(\kappa-1)}}\right)\right),\ s>0.
\end{equation*}
As explained by \cite[Section 3]{RW}, the weak Poinca\'e inequality \eqref{l3-5-1a} always holds
for such $(\E,\D(\E))$ and $\mu$. Hence by Example \ref{ex1-1}, we know that when $\kappa>2$, the SLSI \eqref{e1-2} holds with
associated rate function
\begin{equation*}
\beta_{{\rm SL}}(s)=c_1\left(1+s^{-\frac{\kappa}{\kappa-2}}\right),\ s>0,
\end{equation*}
and when $\kappa\in (0,2]$, the WLSI \eqref{e1-3} holds with following rate function,
\begin{equation*}
\beta_{{\rm WL}}(s)=c_2\left(1+\log^{\frac{2-\kappa}{\kappa}}(1+s^{-1})\right),\ s>0.
\end{equation*}
\end{remark}

\section{Proof of Main Theorems}

The proof of Theorem \ref{t1-1} and \ref{t1-2} will be split into following lemmas.

\begin{lemma}\label{l3-5}
Assume that the  SPI  \eqref{e1-1} holds.
Then the following conclusions are true.
\begin{itemize}
\item [(1)] Suppose that the weak Poincar\'e inequality \eqref{l3-5-1a} is true,
then $\mu$ satisfies the WLSI \eqref{e1-3}
with the rate function $\beta_{{\rm WL}}:\R_+ \to \R_+$ defined by \eqref{l3-5-2a}.
\medskip
\item [(2)] Furthermore, if we assume \eqref{l3-5-2} is true, then the SLSI \eqref{e1-2} holds
with rate function $\beta_{{\rm SL}}(s)$ defined by \eqref{l3-5-3}.
\end{itemize}
\end{lemma}
\begin{proof}
{\bf Step 1} Given $f\in \D(\E)$, choose $\delta>2$ and define
\begin{align}\label{l3-5-6a}
f_n:=\min\left\{\left(|f|-\delta^{\frac{n}{2}}\right)_+, \left(\delta^{\frac{n+1}{2}}-\delta^{\frac{n}{2}}\right)_+\right\},\ \ \forall\ n\ge 0.
\end{align}
According to \cite[Lemma 3.3.2]{30b} and the proof of \cite[Theorem 3.3.3]{30b}
(see also the proof of \cite[Lemma 2.2, Theorem 2.3]{W}), it holds that
\begin{equation}\label{l3-5-4}
\sum_{n=0}^\infty\E(f_n,f_n)\le \E(f,f).
\end{equation}
For every $n\ge 0$, let $A_n:=\{x \in \Omega; \delta^n \le f^2(x)< \delta^{n+1}\}$, $B_n:=\{x\in \Omega; f^2(x)< \delta^n\}$.

Hence by \eqref{e1-1} we have
\begin{align}\label{l3-5-4a}
\int_{\Omega}f_n^2 d\mu \le s\E(f_n,f_n)+\beta_{{\rm SP}}(s)\left(\int_{\Omega}|f_n|d\mu\right)^2,\ s>0.
\end{align}
Note that $f_n(x)\neq 0$ only if $x\in B_n^c$, using Cauchy-Schwarz inequality, it holds
that
\begin{align*}
\left(\int_{\Omega}|f_n|d\mu\right)^2&\le \mu(B_n^c)\left(\int_{\Omega}f_n^2 d\mu\right)
\le \delta^{-n}\left(\int_{\Omega}f_n^2 d\mu\right),
\end{align*}
where we have used the property $\mu(B_n^c)\le \delta^{-n}\mu(f^2)=1$, which follows from the Markov inequality and
the fact $\mu(f^2)=1$. So according to this and \eqref{l3-5-4a} we derive
\begin{align}\label{l3-5-5}
\int_{\Omega}f_n^2 d\mu\le \xi_1(\delta^{-n})\E(f_n,f_n),
\end{align}
where $\xi_1(\cdot)$ is defined by \eqref{t1-2-1}. In particular, by definition we know
that there exists a $n_0\ge 0$ such that $\xi_1(\delta^{-n})>0$ for every $n\ge n_0-1$.

Therefore by definition of $A_n$ we have for every $n\ge n_0$
\begin{equation}\label{l3-5-5a}
\begin{split}
\int_{A_n}f^2\log f^2d\mu&\le \delta^{n+1}\log\left(\delta^{n+1}\right)\mu(A_n)= \frac{\delta^{n+1}\log\left(\delta^{n+1}\right)}
{\left(\delta^{\frac{n}{2}}-\delta^{\frac{n-1}{2}}\right)^2}
\int_{A_n}f_{n-1}^2 d\mu\\
&\le \frac{\delta^{n+1}\log\left(\delta^{n+1}\right)}
{\left(\delta^{\frac{n}{2}}-\delta^{\frac{n-1}{2}}\right)^2}\int_{\Omega}f_{n-1}^2d\mu
\le c_1n\xi_1(\delta^{-n+1})\E(f_{n-1},f_{n-1}).
\end{split}
\end{equation}
Here the second step is due to the fact that $f_{n-1}= \left(\delta^{\frac{n}{2}}-\delta^{\frac{n-1}{2}}\right)^2$
on $A_n$, the last step follows from \eqref{l3-5-5}.

Hence for every $k\ge n_0$ we derive
\begin{equation*}
\begin{split}
\sum_{n=n_0}^k \int_{A_n}f^2\log f^2d\mu&\le
c_1\left(\sup_{n_0\le n\le k}n\xi_1(\delta^{-n+1})\right)\sum_{n=n_0}^{k}\E(f_{n-1},f_{n-1})\\
&\le c_1\left(\sup_{n_0\le n\le k}n\xi_1(\delta^{-n+1})\right)\E(f,f),
\end{split}
\end{equation*}
where the last step is due to \eqref{l3-5-4}.

Meanwhile by direct computation it holds that
\begin{equation*}
\begin{split}
\sum_{n=k+1}^{\infty} \int_{A_n}f^2\log f^2d\mu&\le
\|f\|_\infty^2\cdot \left(\sum_{n=k+1}^{\infty}\log (\delta^{n+1})\mu(A_n)\right)\\
&\le c_2(\delta)\|f\|_\infty^2\left(\sum_{n=k+1}^\infty(n+1)\frac{\mu(f^2)}{\delta^{n}}\right)
\le c_3k\delta^{-k}\|f\|_\infty^2,
\end{split}
\end{equation*}
where in the last step we have used the fact $\mu(f^2)=1$.

Also note that
\begin{align*}
\int_{B_{n_0}}f^2\log f^2d\mu\le n_0\log\delta\int_{\Omega}f^2d\mu.
\end{align*}

Therefore putting all above estimates together, we know that for every $f\in \D(\E)$ with $\mu(f^2)=1$ and every $k\ge n_0$,
\begin{equation}\label{l3-5-6}
\begin{split}
\Ent_{\mu}(f^2)&=\int_{\Omega}f^2\log f^2d\mu=
\int_{B_{n_0}}f^2\log f^2d\mu+\sum_{n=n_0}^\infty\int_{A_n}f^2\log f^2d\mu\\
&\le c_1\left(\sup_{n_0\le n\le k}n\xi_1(\delta^{-n+1})\right)\E(f,f)+c_3k\delta^{-k}\|f\|_\infty^2+c_4\int_{\Omega}f^2d\mu.
\end{split}
\end{equation}
Note that $\Ent_{\mu}(c f)=c\Ent_{\mu}(f)$ for every $c>0$, so taking $\frac{f}{\sqrt{\mu(f^2)}}$ in \eqref{l3-5-6} we know
it still holds for each $f\in \D(\E)$ (without the requirement $\mu(f^2)=1$).

Combining the weak Poinca\'e inequality \eqref{l3-5-1a} with SPI
\eqref{e1-1}, and applying \cite[Proposition 1.3]{RW} we can get the (standard) Poincar\'e inequality
for $\mu$,
\begin{equation}\label{l3-5-7}
\var_{\mu}(f)\le c_5\E(f,f),\ \forall\ f\in \D(\E).
\end{equation}

According to \cite[Proposition 6.1.1]{30b}  we have for every $f\in \D(\E)$,
\begin{align*}
\Ent_{\mu}(f^2)\le \Ent_{\mu}(\hat f^2)+2\mu(\hat f^2),
\end{align*}
where $\hat f:=f-\mu(f)$.

Since $\var_{\mu}(f)=\mu(\hat f^2)$, applying \eqref{l3-5-6}, \eqref{l3-5-7}, and using \eqref{e1-2a} we get for every $k\ge n_0$,
\begin{align*}
\Ent_{\mu}(f^2)&\le \Ent_{\mu}(\hat f^2)+2\mu(\hat f^2)\\
&\le \left(c_1\left(\sup_{n_0\le n\le k}n\xi_1(\delta^{-n+1})\right)+c_5(1+c_4)\right)\E(f,f)+2c_3k\delta^{-k}\|f\|_\infty^2.
\end{align*}
By this inequality we can prove the WLSI with desired rate function $\beta_{{\rm WL}}(s)$
defined by \eqref{l3-5-2a}.

{\bf Step 2} Now we assume that \eqref{l3-5-2} holds. Then it is easy to verify that
$N_0(s)$ is well-defined by \eqref{l3-5-3aa} (the set to take infimum is not empty) if
$s\in (0,s_0]$ for some $s_0>0$ (small enough).

Hence by \eqref{l3-5-5a} we have
\begin{equation*}
\int_{A_n}f^2\log f^2d\mu\le s\E(f_{n-1},f_{n-1}),\ \ \forall\ n\ge N_0(s),
\end{equation*}
which implies that
\begin{equation*}
\sum_{n=N_0(s)+1}^\infty \int_{A_n}f^2\log f^2d\mu\le s\sum_{n=N_0(s)+1}^\infty\E(f_{n-1},f_{n-1})
\le s\E(f,f),
\end{equation*}
where the last step follows from \eqref{l3-5-4}.

By direct computation,
\begin{equation*}
\begin{split}
\int_{B_{N_0(s)}}f^2\log f^2d\mu&\le \log \delta(1+N_0(s))\int_{\Omega}f^2d\mu.
\end{split}
\end{equation*}
Combining all above estimates yields that for every
$f\in \D(\E)$ with $\mu(f^2)=1$,
\begin{align*}
\Ent_{\mu}(f^2)&=\int_{\Omega}f^2\log f^2d\mu
=\int_{B_{N_0(s)}}f^2\log f^2d\mu+\sum_{n=N_0(s)+1}^\infty\int_{A_n}f^2\log f^2d\mu\\
&\le s\E(f,f)+\log \delta(1+N_0(s))
\int_{\Omega}f^2d\mu.
\end{align*}
So we have established SLSI with rate function defined by \eqref{l3-5-3}.
\end{proof}

\begin{lemma}\label{l3-1}
Suppose that WLSI \eqref{e1-3} holds and there exist $\theta>0$, $\delta\ge 2$ such that
\eqref{l3-1-1} is true, then the SPI \eqref{e1-1} holds with rate function defined by \eqref{l3-1-3}.
\end{lemma}

\begin{proof}
Choosing $\delta>2$, as before for every $f\in \D(\E)$ we define $\{f_n\}_{n\ge 0}$ by \eqref{l3-5-6a}.
And we set $A_n:=\{x \in \Omega; \delta^n \le f^2(x)< \delta^{n+1}\}$, $B_n:=\{x\in \Omega; f^2(x)< \delta^n\}$
for every $n\ge 0$.

Without loss of generality we assume that $\mu(f^2)=1$, then we can verify that
\begin{equation*}
\mu(B_n^c)\le \delta^{-n}\mu(f^2)=\delta^{-n},\ n\ge 1.
\end{equation*}

It holds that
\begin{align}\label{l3-4-7}
\int_{A_{n+1}}f^2d\mu&\le \delta^{n+1}\mu(A_{n+1})=\frac{\delta^{n+1}}{\left(\delta^{\frac{n+1}{2}}-\delta^{\frac{n}{2}}\right)^2}\int_{A_{n+1}}f_n^2d\mu,
\end{align}
where we use the fact $f(x)=\delta^{\frac{n+1}{2}}-\delta^{\frac{n}{2}}$ for every $x\in A_{n+1}$.

According to the variational formula of entropy (see e.g. 3rd line in the proof of \cite[Theorem 2.1]{WLSI}) we have
\begin{equation}\label{l3-4-3}
\Ent_{\mu}(f^2)\ge \int_{\Omega}f^2\varphi d\mu,\ \ {\rm every}\ \varphi:\Omega \to \R\ {\rm such\ that}\
\int_{\Omega}e^{\varphi}d\mu\le 1.
\end{equation}
Let
\begin{equation}\label{l3-4-3a}
\varphi_n=
\begin{cases}
n\log \delta\ &\ x\notin B_n,\\
-\infty,\ &\ x\in B_n,
\end{cases}
\end{equation}
so $\int_{\Omega}e^{\varphi_n}d\mu\le \delta^{n}\mu(B_n^c)\le 1$.

Since $f_n(x)\neq 0$ only if $x\notin B_n$,
putting such $\varphi_n$ in \eqref{l3-4-3} we obtain that
\begin{equation*}
\begin{split}
\mu(f_n^2)&\le \frac{1}{n\log \delta}\Ent_{\mu}(f_n^2)\le \frac{1}{n\log \delta}
\left(\beta_{{\rm WL}}(r)\E(f_n,f_n)+r\|f_n\|_\infty^2\right)\\
&\le \frac{1}{n\log \delta}
\left(\beta_{{\rm WL}}(r)\E(f_n,f_n)+r\delta^{n}\left(\delta^{1/2}-1\right)^2\right),\ r>0,
\end{split}
\end{equation*}
where in the second step we have applied WLSI \eqref{e1-3}, and the last step is due
to the fact that $\|f_n\|_\infty^2\le \delta^{n}\left(\delta^{1/2}-1\right)^2$.

Therefore taking $r=r_n:=\delta^{-n}n^{-\theta}$ and using \eqref{l3-4-7} we obtain that for all $k\ge 2$,
\begin{equation*}
\begin{split}
\sum_{n=k}^{\infty}\int_{A_{n+1}}f^2d\mu&\le \sum_{n=k}^\infty \frac{c_0\beta_{{\rm WL}}\left(\delta^{-n}n^{-\theta}\right)}{n}\E(f_n,f_n)
+\sum_{n=k}\frac{c_0}{n^{1+\theta}}\\
&\le c_0\left(\sup_{n\ge k}\frac{\beta_{{\rm WL}}\left(\delta^{-n}n^{-\theta}\right)}{n}\right)\E(f,f)+c_1k^{-\theta},
\end{split}
\end{equation*}
where we have used \eqref{l3-5-4} in the last step. Therefore there exists an $n_0\ge 0$ so that $c_1n_0^{-\theta}\le \frac{1}{4}$ and we obtain
that
\begin{equation}\label{l3-1-2}
\begin{split}
\sum_{n=k}^{\infty}\int_{A_{n+1}}f^2d\mu&\le c_0\left(\sup_{n\ge k}\frac{\beta_{{\rm WL}}\left(\delta^{-n}n^{-\theta}\right)}{n}\right)\E(f,f)+
\frac{1}{4}\\
&=c_0\left(\sup_{n\ge k}\frac{\beta_{{\rm WL}}\left(\delta^{-n}n^{-\theta}\right)}{n}\right)\E(f,f)+
\frac{\mu(f^2)}{4},\ \ \forall\ k\ge n_0,
\end{split}
\end{equation}
where in the last step we have used the fact $\mu(f^2)=1$.

Let $\hat f_{k}(x):=\min\{f(x),\delta^{\frac{k+1}{2}}\}$, according to \cite{BR} or \cite[Lemma 2, Lemma 3]{BCR1}, we get for every $r\ge 1$,
\begin{align*}
\mu(\hat f_k^2)-r\mu(|\hat f_k|)^2&\le
\mu(|\hat f_k-m_k|^2)-(r-1)\mu(|\hat f_k-m_k|)^2\\
&\le \sup\left\{\int_\Omega (\hat f_k-m_k)^2g d\mu; g\in \mathscr{G}_r\right\}
\end{align*}
where $m_k$ is the median of $\hat f_k$ and
\begin{align*}
\mathscr{G}_r:=\left\{g:\Omega\to [0,1); \int_\Omega \frac{1}{1-g}d\mu\le 1+\frac{1}{r-1}\right\}.
\end{align*}
Since by definition $(\hat f_k-m_k)_+\le 2\delta^{\frac{k+1}{2}}$, it holds that
\begin{align*}
\sup\left\{\int_\Omega (\hat f_k-m_k)^2g d\mu; g\in \mathscr{G}_r\right\}
&\le 2\delta^{k+1}\sup\left\{\int_\Omega g d\mu; g\in \mathscr{G}_r\right\}\\
&\le 2\delta^{k+1}\mu(\Omega)\left(1-\left(1+\frac{1}{(r-1)\mu(\Omega)}\right)^{-1}\right)\\
&\le \frac{c_2\delta^{k+1}}{r-1},\ \ r\ge 1,
\end{align*}
where the second step is due to \cite{BR} or \cite[Lemma 4]{BCR1}
(with $a=\frac{1}{2}$, $K=1+\frac{1}{r-1}$, $A=X=\Omega$).

Then putting all above estimates we arrive at that for every $k\ge n_0$ and $r\ge 1$
\begin{align*}
\int_{B_{k+1}}f^2d\mu&\le \int_{\Omega}\hat f_{k}^2d\mu
\le \frac{c_3\delta^{k}}{r-1}+r\mu(|\hat f_k|)^2\le
\frac{c_3\delta^{k+1}}{r-1}\mu(f^2)+r\mu(|f|)^2,
\end{align*}
where we have used the fact $\mu(f^2)=1$. Hence taking $r=1+4c_3\delta^{k+1}$ we obtain
\begin{align}\label{l3-1-3a}
\int_{B_{k+1}}f^2d\mu\le \frac{\mu(f^2)}{4}+c_4\delta^k\mu(|f|)^2.
\end{align}
This, together with \eqref{l3-1-2} yields that for every $k\ge n_0$,
\begin{align*}
\mu(f^2)&=\int_{B_{k+1}}f^2d\mu+\sum_{n=k}^\infty\int_{A_{n+1}}f^2d\mu\\
&\le c_5\left(\sup_{n\ge k}\frac{\beta_{{\rm WL}}\left(\delta^{-n}n^{-\theta}\right)}{n}\right)\E(f,f)+c_5
\delta^{k}\mu(|f|)^2+\frac{\mu(f^2)}{2},
\end{align*}
which implies
\begin{align*}
\mu(f^2)\le 2c_5\left(\sup_{n\ge k}\frac{\beta_{{\rm WL}}\left(\delta^{-n}n^{-\theta}\right)}{n}\right)\E(f,f)+2c_5
\delta^{k}\mu(|f|)^2,\ \ \forall\ k\ge n_0,
\end{align*}
by this we can prove immediately the SPI \eqref{e1-1} with rate function defined by \eqref{l3-1-3}. 

\end{proof}

\begin{lemma}\label{l3-4}
Suppose that $\mu$ satisfies SLSI \eqref{e1-2},
then the SPI \eqref{e1-1} holds with the rate function $\beta_{{\rm SP}}(s)$ defined by
\eqref{l3-4-1}.
\end{lemma}

\begin{proof}
Choosing a $\delta>2$, for every $f\in \D(\E)$ we define $\{f_n\}_{n\ge 0}$ according to
\eqref{l3-5-6a}. And sill let $A_n:=\{x \in \Omega; \delta^n \le f^2(x)< \delta^{n+1}\}$, $B_n:=\{x\in \Omega; f^2(x)< \delta^n\}$
for every $n\ge 0$.

Without loss of generality we assume that $\mu(f^2)=1$, so it holds that
\begin{equation*}
\mu(B_n^c)\le \delta^{-n}\mu(f^2)=\delta^{-n},\ n\ge 1.
\end{equation*}

As explained in the proof of Lemma \ref{l3-1} above, we can take $\varphi_n$ as defined in
\eqref{l3-4-3a} and apply \eqref{l3-4-3} to obtain that
\begin{equation}\label{l3-4-4a}
\begin{split}
\mu(f_n^2)&\le \frac{1}{n\log \delta}\Ent_{\mu}(f_n^2)\le \frac{1}{n\log \delta}\left(r\E(f_n,f_n)+\beta_{{\rm SL}}(r)\mu(f_n^2)\right),\ r>0,
\end{split}
\end{equation}
where the last step is due to \eqref{e1-2}.
Therefore we have
\begin{equation}\label{l3-4-4}
\mu(f_n^2)\le \inf_{r>0}\left(\frac{r}{n\log \delta-\beta_{{\rm SL}}(r)}\right)\E(f_n,f_n)=\xi_2(n \log \delta)\E(f_n,f_n).
\end{equation}
where $\xi_2(\cdot)$ is defined by \eqref{t1-2-1a}. By definition we know that there exists
an $n_0\ge 1$ such that $\xi_2(n \log \delta)>0$ for every $n\ge n_0$.

Combining this with \eqref{l3-4-7} yields that for every $k\ge n_0$
\begin{equation}\label{l3-4-6}
\begin{split}
\sum_{n=k}^{\infty}\int_{A_{n+1}}f^2d\mu&\le \sum_{n=k}^\infty c_0\xi_2(n \log \delta)\E(f_n,f_n)
\le c_1\xi_2(k \log \delta)\E(f,f),
\end{split}
\end{equation}
where we have used \eqref{l3-5-4} and the fact that $n \mapsto \xi_2(n \log \delta)$ is non-increasing.

Then combining this with \eqref{l3-1-3a}(which is still true in this case) yields that
for every $k\ge n_0$,
\begin{align*}
\mu(f^2)&=\int_{B_{k+1}}f^2d\mu+\sum_{n=k}^\infty\int_{A_{n+1}}f^2d\mu\\
&\le c_2\xi_2(k \log \delta)\E(f,f)+c_2
\delta^{k}\mu(|f|)^2+\frac{\mu(f^2)}{2}.
\end{align*}
This implies that
\begin{align*}
\mu(f^2)\le 2c_2\xi_2(k \log \delta)\E(f,f)+2c_2
\delta^{k}\mu(|f|)^2,\ k\ge k_0.
\end{align*}
Based on this we can prove that SPI \eqref{e1-1} holds with the desired rate function \eqref{l3-4-1}.
\end{proof}

\begin{proof}[Proof of Theorem \ref{t1-1} and Theorem \ref{t1-2}]
The conclusion of Theorem \ref{t1-1} is a direct consequence of Lemma \ref{l3-5} (1) and
Lemma \ref{l3-1}. And by Lemma \ref{l3-5}(2) and Lemma \ref{l3-4} we can prove Theorem \ref{t1-2} immediately.
\end{proof}

\begin{proof}[Proof of Example \ref{ex1-1}]
(1) Suppose SPI \eqref{e1-1} holds with rate function \eqref{ex1-1-1}. By definition \eqref{t1-2-1} we know that for some $t_0>0$,
\begin{equation*}
c_1\log^{-\frac{1}{\theta}}(1+t^{-1})\le \xi_1(t)\le c_2\log^{-\frac{1}{\theta}}(1+t^{-1}),\ t\le t_0,
\end{equation*}
which implies that \eqref{l3-5-2} is true when $\theta\in [1/2,1)$. By this and \eqref{l3-5-3} we get the desired rate function
\eqref{ex1-1-2} for SLSI. When $\theta\in [1,+\infty)$, applying \eqref{l3-5-2a} directly we can prove
WLSI with rate function \eqref{ex1-1-3}.

\medskip

(2) If SLSI holds with rate function \eqref{ex1-1-2} for some $\theta\in [1/2,1)$, by definition \eqref{t1-2-1a} we can find
a $t_1>0$ so that
\begin{equation*}
c_3t^{-\frac{1}{\theta}}\le \xi_2(t) \le c_4t^{-\frac{1}{\theta}},\ t\ge t_1.
\end{equation*}
This, along with \eqref{l3-4-1} implies the SPI with rate function \eqref{ex1-1-1}. Meanwhile
if WLSI with rate function \eqref{ex1-1-3} for some $\theta\in [1,+\infty)$ is true,
it is easy to verify the validity of \eqref{l3-1-1}, then applying \eqref{l3-1-3} directly we
can establish SPI with the associated rate function \eqref{ex1-1-1}.
\end{proof}

\bibliographystyle{plain}

\vskip 0.3truein



\begin{thebibliography}{10}


\bibitem{BCG08}
D.~Bakry, P.~Cattiaux, and A.~Guillin.
\newblock Rate of convergence for ergodic continuous {M}arkov processes:
  {L}yapunov versus {P}oincar{\'e}.
\newblock {\em J. Funct. Anal.}, 254(3):727--759, 2008.

\bibitem{B85}
D.~Bakry and M.~{\'E}mery.
\newblock Diffusions hypercontractives.
\newblock In {\em S\'eminaire de Probabiliti\'es {XIX} 1983/84}, volume 1123 of
  {\em Lect. Notes Math.}, pages 177--206. Springer, Berlin, 1985.




\bibitem{B05}
F.~Barthe, P.~Cattiaux, and C.~Roberto.
\newblock Concentration for independent random variables with heavy tails.
\newblock {\em AMRX Appl. Math. Res. Express}, (2):39--60, 2005.


\bibitem{BCR}
F. Barthe, P. Cattiaux and C. Roberto.
\newblock Interpolated inequalities between exponential and Gaussian, Orlicz hypercontractivity and isoperimetry.
\newblock {\em Rev. Mat. Iberoam.}, {\bf 22}(3): 993--1067, 2006.

\bibitem{BCR1}
F. Barthe, P. Cattiaux and C. Roberto.
\newblock Isoperimetry between exponential and Gaussian.
\newblock {\em Electron. J. Probab.} {\bf 12}: 1212--1237, 2007.

\bibitem{BR}
F. Barthe and C. Roberto.
\newblock Sobolev inequalities for probability measures on the real line.
\newblock {\em Studia Math.} {\bf 123}(1): 81--95, 1997.


\bibitem{BM}
M. Biroli and P. Maheux.
\newblock On equivalence of super log-Sobolev and Nash type inequalities.
\newblock {\em Colloq. Math.} {\bf 137}(2): 189--208, 2014.





\bibitem{Cao23}
Y.~Cao, J.~Lu, and L.~Wang.
\newblock On explicit {$L^2$}-convergence rate estimate for underdamped
  {L}angevin dynamics.
\newblock {\em Arch. Ration. Mech. Anal.}, 247(5):90, 2023.

\bibitem{WLSI}
P.~Cattiaux, I.~Gentil, and A.~Guillin.
\newblock Weak logarithmic {S}obolev inequalities and entropic convergence.
\newblock {\em Probab. Theory Relat. Fields}, 139(3-4):563--603, 2007.

\bibitem{CG}
P. Cattiaux and A. Guillin.
\newblock Hitting times, functional inequalities, Lyapunov conditions and uniform ergodicity.
\newblock {\em J. Funct. Anal.}, {\bf 272}(6): 2361--2391, 2017.

\bibitem{entropy}
P. Cattiaux, A. Guillin, P. Monmarch\'e, and C. Zhang.
\newblock Entropic multipliers method for {L}angevin diffusion and weighted log-Sobolev inequalities. \newblock {\em J. Funct. Anal.}, 277(11):108288, 2019.






\bibitem{DS84} E. B. Davies and B. Simon.
\newblock Ultracontractivity and the heat kernel for Schr\"odinger operators and Dirichlet Laplacians.
\newblock {\em J. Funct. Anal.}, {\bf 59}(2): 335--395, 1984






\bibitem{FOT}
M. Fukushima, Y. Oshima, and M. Takeda.
\emph{Dirichlet forms and symmetric Markov processes,}
De Gruyter Stud. Math., 19
Walter de Gruyter, Berlin, 1994, x+392 pp.

\bibitem{GW}
F.-Z. Gong and F.-Y. Wang.
\newblock Functional inequalities for uniformly integrable semigroups and application to essential spectrums.
\newblock {\em Forum Math.} {\bf 14}(2): 293--313, 2002.


\bibitem{G}
L. Gross.
\newblock Logarithmic Sobolev inequalities.
\newblock {\em Amer. J. Math.} 97(4): 1061--1083, 1975.

\bibitem{GW19}
M.~Grothaus and F.-Y. Wang.
\newblock Weak {P}oincar\'e inequalities for convergence rate of degenerate
  diffusion processes.
\newblock {\em Ann. Probab.}, 47(6):2930--2952, 2019.

\bibitem{GLM}
A. Guillin, P. Le Bris, and P. Monmarch\'e.
\newblock Uniform in time propagation of chaos for the 2D vortex model and other singular stochastic systems.
\newblock {\em J. Eur. Math. Soc.}, 27(6): 2359--2386, 2025.

\bibitem{GZ}
A. Guionnet and B. Zegarlinski.
\newblock {\em Lectures on logarithmic Sobolev inequalities},
Lecture Notes in Math., {\bf 1801}
Springer-Verlag, Berlin, 1--134, 2003.







\bibitem{Le}
M. Ledoux.
\newblock {\em Concentration of measure and logarithmic Sobolev inequalities},
Lecture Notes in Math., {\bf 1709}
Springer-Verlag, Berlin, 120--216, 1999.

\bibitem{MJT} D. Matthes, A. J\"ungel and G. Toscani.
\newblock Convex Sobolev inequalities derived from entropy dissipation.
\newblock {\em Arch. Ration. Mech. Anal.}, {\bf 199}(2): 563--596, 2011.





\bibitem{RW}
M.~R\"ockner and F.-Y. Wang.
\newblock Weak {P}oincar\'e inequalities and {$L^2$}-convergence rates of
  {M}arkov semigroups.
\newblock {\em J. Funct. Anal.}, 185(2):564--603, 2001.





\bibitem{W}
F.-Y. Wang.
\newblock Functional inequalities for empty essential spectrum.
\newblock {\em J. Funct. Anal.}, 170(1):219--245, 2000.

\bibitem{W1}
F.-Y. Wang.
\newblock Functional inequalities and spectrum estimates: the infinite measure case.
\newblock {\em J. Funct. Anal.}, 194(2): 288--310, 2002.

\bibitem{30b}
F.-Y. Wang.
\newblock {\em Functional Inequalities, Markov Semigroups and Spectral Theory}.
\newblock Science Press, Beijing/New York, 2005.

\bibitem{WW15}
F.-Y. Wang. and J. Wang.
\newblock Functional inequalities for stable-like Dirichlet forms,
\newblock {\em J. Theoret. Probab.}, {\bf 28}(2): 423--448, 2015.


\end{thebibliography}

\vskip 0.3truein

{\bf Xin Chen:}
Shanghai Jiao Tong University, 200240, Shanghai, P.R. China.
E-mail: \texttt{chenxin217@sjtu.edu.cn}

\medskip
{\bf Qiuchen Yang:}
Shanghai Jiao Tong University, 200240, Shanghai, P.R. China.
E-mail: \texttt{yangqiuchen@sjtu.edu.cn}

\end{document}